\documentclass[11pt,twoside]{amsart}

\usepackage{amssymb}
\usepackage{latexsym}
\usepackage{graphicx}

\theoremstyle{plain}
\newtheorem{theorem}{Theorem}[section]

\newtheorem{proposition}[theorem]{Proposition}
\newtheorem{lemma}[theorem]{Lemma}
\newtheorem{corollary}[theorem]{Corollary}
\newtheorem{definition}[theorem]{Definition}
\newtheorem{remark}[theorem]{Remark}

\newcommand{\eps}{\varepsilon}
\newcommand{\dopu}{{:}\allowbreak\ } 
\newcommand{\R}{{\mathbb R}}
\newcommand{\N}{{\mathbb N}}

\newcounter{abc} 
\newcounter{iiiii} 

\newenvironment{aequivalenz}
{\setcounter{iiiii}{0}
\begin{list}
{{\rm (\roman{iiiii})}}
{\usecounter{iiiii}
\parsep=0pt plus 1pt
\topsep=1pt plus 2pt minus 1pt
\itemsep=1pt plus 2pt minus 1pt
\leftmargin=3\baselineskip \labelsep=.6\baselineskip
\labelwidth=2.4\baselineskip
\rightmargin 0pt}
}
{\end{list}}

\newenvironment{statements}
{\setcounter{abc}{0}
\begin{list}
{{\rm (\alph{abc})}}
{\usecounter{abc}
\parsep=0pt plus 1pt
\topsep=1pt plus 2pt minus 1pt
\itemsep=1pt plus 2pt minus 1pt
\leftmargin=3\baselineskip \labelsep=.6\baselineskip
\labelwidth=2.4\baselineskip
\rightmargin 0pt}
}
{\end{list}}

\numberwithin{equation}{section}

\title[Daugavet centers]{Daugavet centers and direct sums of Banach spaces}
\author{Tetiana V. Bosenko}
\subjclass[2000]{Primary 46B04; secondary  46B20, 46B40}
\keywords{Daugavet center, Daugavet property, direct sum of Banach spaces}
\address{Department of Mechanics and Mathematics, V.N.~Karazin Kharkiv National University, 
pl.~Svobody~4,  61077~Kharkiv, Ukraine}
\email{t.bosenko@mail.ru}

\begin{document}
\sloppy

\begin{abstract}
A linear continuous nonzero operator $G \dopu X \to Y$ is a Daugavet center if every rank-$1$ operator $T\dopu X \to Y$ satisfies $\|G+T\|=\|G\|+\|T\|$. We study the case when either $X$ or $Y$ is a sum $X_1 \oplus_F X_2$ of two Banach spaces $X_1$ and $X_2$ by some two-dimensional Banach space $F$. We completely describe the class of those $F$ such that for some spaces $X_1$ and $X_2$ there exists a Daugavet center acting from $X_1\oplus_F X_2$, and the class of those $F$ such that for some pair of spaces $X_1$ and $X_2$ there is a Daugavet center acting into $X_1\oplus_F X_2$. We also present several examples of such Daugavet centers.
\end{abstract}

\maketitle

\section{Introduction}
In the present paper we consider real Banach spaces which do not equal $\{0\}$, and denote them $E$, $X$ or $Y$.
A linear continuous nonzero operator $G \dopu X \to Y$ is called a \textit{Daugavet center}~\cite{BosKad} if every rank-$1$ operator $T\dopu X \to Y$ satisfies the equation
\begin{equation}
\|G+T\|=\|G\|+\|T\|.
\label{eqDC} 
\end{equation}
\begin{definition}
We say that $X$ is a \textit{Daugavet domain} if there exists a Daugavet center $G \dopu X \to Y$ for some $Y$, and is a \textit{Daugavet range} if there is a Daugavet center $G \dopu E \to X$ for some $E$.
\label{DD}
\end{definition}

Throughout this paper $F=(\R^2, \|\cdot\|)$ with $\|(1,0)\|=\|(0,1)\|=1$ and
\begin{equation}
\|(a_1,a_2)\|=\|(|a_1|,|a_2|)\|
\label{ModulNormF}
\end{equation} 
for every $(a_1,a_2)\in F$. For Banach spaces $X_1$ and $X_2$ their $F$-sum $X_1 \oplus_F X_2$ is the space of all pairs $(x_1,x_2)$ where $x_1\in X_1$ and $x_2\in X_2$, $\|(x_1,x_2)\|:=\|(\|x_1\|,\|x_2\|)\|_F$.  

We introduce the following order on $F$: $(a_1, a_2)\geq (b_1, b_2)$ if $a_1\geq b_1$ and $a_2\geq b_2$. It follows from~(\ref{ModulNormF}) and a convexity argument that
for every $(a_1,a_2),(b_1,b_2)\in F$ with $(|a_1|,|a_2|)\leq(|b_1|,|b_2|)$ the inequality $\|(a_1,a_2)\|\leq\|(b_1,b_2)\|$ holds true. In this partial order $F$ is a Banach lattice~\cite{LiTz2}, so we will use the term ``two-dimensional lattice" for $F$ in the sequel.

The problem which we solve in this paper, consists of two parts: first, we characterize the class of those $F$ for which there exist $X_1$ and $X_2$ such that $X_1 \oplus_F X_2$ is a Daugavet domain, and secondly, we characterize the class of those $F$ for which there are $X_1$ and $X_2$ such that $X_1 \oplus_F X_2$ is a Daugavet range. 

Remark that a Daugavet domain and a Daugavet range are generalizations of a Banach space with the Daugavet property, and this motivates our interest in the subject. A Banach space $X$ is said to have the Daugavet property if the identity operator $\mathrm{Id} \dopu X \to X$ is a Daugavet center. The study of spaces with the Daugavet property is a rapidly developing branch of Banach space theory (see~\cite{KadSSW},~\cite{Shv1},~\cite{dirk-irbull}, and the most recent developments in~\cite{KadShepW},~\cite{KW}).
The following classical spaces have the Daugavet property: $C(K)$ where K is a compact without isolated points~\cite{Daug}, $L_{1}(\mu)$ and $L_\infty(\mu)$ where $\mu$ has no atoms~\cite{Loz}, and many Banach algebras~(\cite{Dirk10},~\cite{Woj92}). Some exotic spaces have the Daugavet property as well, for instance, Talagrand's space~(\cite{KadSSW},~\cite{Tal}) and Bourgain-Rosenthal's space~(\cite{BourRos},~\cite{KW}). 

Let us recall some recent results~\cite{BosKad} related to Daugavet centers. If $G$ is a Daugavet center then~(\ref{eqDC}) also holds true when T is a strong Radon-Nikod\'ym operator, e.g., a weakly compact operator. If $X$ is a Daugavet domain or a Daugavet range then $X$ contains subspaces isomorphic to $\ell_{1}$, is non-reflexive and does not have an unconditional basis (countable or uncountable). One cannot even embed such an $X$ into a space having an unconditional basis or having a representation as unconditional sum of reflexive subspaces. 
In~\cite{Pop} Popov proves that every isometric embedding of $L_1[0,1]$ into itself is a Daugavet center. However, in~\cite{BosKad} one can find examples of Daugavet centers which are not isometries.

The present work is inspired by~\cite{BKSW}. It was shown in~\cite{BKSW} and~\cite{KadSSW} that if $X_1$ and $X_2$ have the Daugavet property and $F=\ell_1^{\scriptscriptstyle (2)}$
or $\ell_\infty^{\scriptscriptstyle (2)}$ then $X_1 \oplus_F X_2$ has the Daugavet
property as well.
 In~\cite{BKSW} the authors prove that $X_1 \oplus_F X_2$ has the Daugavet property only if $F=\ell_1^{\scriptscriptstyle (2)}$ or $\ell_{\infty}^{\scriptscriptstyle (2)}$.
In our paper we generalize these results of~\cite{BKSW}, but use a new approach to the problem.
Surprisingly in both parts of our problem we discover other spaces apart from $F=\ell_1^{\scriptscriptstyle (2)}$ and $F=\ell_{\infty}^{\scriptscriptstyle (2)}$, which satisfy our demands.

Our approach is based on a necessary condition for a general Banach space $X$ to be a Daugavet domain and on a necessary condition for $X$ to be a Daugavet range. We deduce these two conditions in Section 2 (see Definition~\ref{Deny}, Lemma~\ref{NoDCfromDeny} and Definition~\ref{DenyAst}, Lemma~\ref{NoDCtoDenyAst}) and then we show in Section 3 how they depend on $F$ when $X=X_1\oplus_F X_2$ (see Lemma~\ref{CorFDeny} and Lemma~\ref{CorFDenyAst}).

In Section 4 we find a rather small class $\mathfrak{N}_2$ such that if $X_1 \oplus_F X_2$ is a Daugavet range then $F \in \mathfrak{N}_2$, and in Section 5 we discover the analogous class $\mathfrak{M}_2$ for the case of a Daugavet domain. Then for every $F\in \mathfrak{M}_2$ we present an example of a Daugavet center acting \textit{from} a sum of two Banach spaces by $F$ (see Proposition~\ref{DCfromExmp}), and this solves the first part of our problem. In a very similar way we solve its second part, namely we give examples of Daugavet centers acting \textit{into} a sum of two Banach spaces by $F$ for every $F\in \mathfrak{N}_2$
(see Proposition~\ref{DCtoExmp}). The obtained results illustrate that the notions of a Daugavet domain and a Daugavet range do not refer to the same Banach spaces.

Throughout this paper $B_X$ denotes the closed unit ball of $X$ and $S_X$ denotes its unit sphere. 
We use the notation
$$
B_F^+:=\{a\in B_F\dopu a\geq 0\}
$$
for the positive part of the unit ball and
$$
S_F^+:=\{a\in S_F\dopu a\geq 0\}
$$
for the positive part of the unit sphere of $F$.
We denote
$$
S(B_X,z^*,\eps):= \{x\in B_X\dopu z^*(x)>1-\eps \}
$$
the slice of $B_X$ determined by $z^*\in S_{X^*}$ and $\eps>0$.
$$
S(B_{X^*},z,\eps)= \{x^*\in B_{X^*} \dopu x^*(z)>1-\eps \}
$$ 
denotes the weak$^*$ slice of $B_{X^*}$ determined by $z\in S_{X}$ and $\eps>0$. For an $x^*\in X^*$ and a $y\in Y$ the symbol $x^*\otimes y$ stands for the operator which acts from $X$ into $Y$ as follows: $(x^*\otimes y)(x)=x^*(x)y$.

Finally, let us cite a fact that we frequently use in the sequel.
\begin{theorem}[\cite{BosKad}, Theorem 2.1]
For an operator $G \dopu X \to Y$ with $\|G\|=1$ the following
assertions are equivalent:
\begin{aequivalenz}
\item $G$ is a Daugavet center.

\item For every $y_0\in S_Y$, $x_0^* \in S_{X^*}$ and $\varepsilon >0$
there is an $x \in S(B_X,x_0^*, \eps)$ with $\|Gx+y_0\|>2-\varepsilon$.

\item For every $y_0\in S_Y$, $x_0^* \in S_{X^*}$ and $\varepsilon >0$
there is a $y^* \in S(B_{Y^*},y_0, \varepsilon)$ with $\|G^*y^*+x_0^*\|>2-\varepsilon$.
\end{aequivalenz}
\label{charDC}
\end{theorem}

\section{Banach spaces denying the Daugavet property}

\begin{definition}
We say that $X$ \textit{denies the Daugavet property with a set} $A \subset S_X$ if there is an $\eps>0$ such that for every $y \in A$ there exists an $x^* \in S_{X^*}$ satisfying
\begin{equation}
\|\mathrm{Id}+x^*\otimes y\|<2-\eps.
\label{DaugEqFail}
\end{equation}
\label{Deny}
\end{definition}

\begin{definition}
We say that $X$ \textit{star-denies the Daugavet property with a set} $A \subset S_{X^*}$ if there is an $\eps>0$ such that for every $x^* \in A$ there exists a $y \in S_{X}$ satisfying~(\ref{DaugEqFail}).
\label{DenyAst}
\end{definition}

\begin{lemma}
For $A \subset S_X$ the following assertions are equivalent:
\begin{aequivalenz}
\item $X$ denies the Daugavet property with A.
\item There is an $\eps>0$ such that for every $y \in A$ a functional $x^* \in S_{X^*}$ may be chosen so that every $x \in S(B_{X},x^*,\eps)$ fulfills $\|x+y\|<2-\eps$. 
\item There is an $\eps>0$ such that for every $y \in A$ a functional $x^* \in S_{X^*}$ may be chosen so that every $y^* \in S(B_{X^*},y,\eps)$ fulfills $\|x^*+y^*\|<2-\eps$.
\end{aequivalenz}
\label{CharDeny}
\end{lemma}

\begin{proof}
$(i)\Rightarrow(ii)$ We have that there is an $\eps>0$ such that for every $y \in A$ there exists an $x^* \in S_{X^*}$ satisfying
$$
\|\mathrm{Id}+x^*\otimes y\|=\sup_{x\in B_X} \|x+x^*(x)y\|<2-\eps.
$$
Hence $\|x+x^*(x)y\|<2-\eps$ for every $x\in B_X$. Let $x\in S(B_{X},x^*, \eps/2)$ then
$$\|x+y\| \leq \|x+x^*(x)y\|+\|y-x^*(x)y\|<2-\eps+|1-x^*(x)|\cdot \|y\|<2-\eps+\frac{\eps}{2}=2-\frac{\eps}{2}$$
which implies $(ii)$.

$(ii)\Rightarrow(i)$ Let $\eps$ and $x^*$ be from $(ii)$.
It is sufficient to show that $\|x+x^*(x)y\|\leq 2-\eps/2$ for every $x\in B_X$. Let $x\in S(B_{X},x^*,\eps/2)$ then 
$$\|x+x^*(x)y\|\leq \|x+y\|+\|y-x^*(x)y\|<2-\eps+|1-x^*(x)|\cdot \|y\|<2-\frac{\eps}{2}.$$
Let $x\in S(B_{X},-x^*,\eps/2)$ then $-x\in S(B_{X},x^*,\eps/2)$. Hence $\|x-y\|<2-\eps$ and
$$\|x+x^*(x)y\|\leq \|x-y\|+\|y+x^*(x)y\|<2-\eps+|1+x^*(x)|\cdot \|y\|<2-\frac{\eps}{2}.$$
Finally, let $x\in B_X\setminus\bigl(S(B_{X},x^*,\eps/2)\cup S(B_{X},-x^*,\eps/2)\bigr)$ then
$$
\|x+x^*(x)y\|\leq \|x\|+|x^*(x)|\cdot \|y\|\leq 2-\frac{\eps}{2}.
$$
The equivalence $(i)\Leftrightarrow(iii)$ can be proved in a very similar fashion to $(i)\Leftrightarrow(ii)$ using the fact that the norms of an operator and of its adjoint coincide.
\end{proof}

\begin{lemma} 
For $A \subset S_{X^*}$ the following assertions are equivalent:
\begin{aequivalenz}
\item $X$ star-denies the Daugavet property with A.
\item There is an $\eps>0$ such that for every $x^*\in A$ a vector $y\in S_{X}$ may be chosen so that every $x\in S(B_{X},x^*,\eps)$ fulfills $\|x+y\|<2-\eps$.
\item There is an $\eps>0$ such that for every $x^*\in A$ a vector $y\in S_{X}$ may be chosen so that every $y^*\in S(B_{X^*},y,\eps)$ fulfills $\|x^*+y^*\|<2-\eps$.
\end{aequivalenz}
\label{CharDenyAst}
\end{lemma}

One can prove Lemma~\ref{CharDenyAst} the same way as Lemma~\ref{CharDeny}. 
The following two lemmas form the main result of this section.

\begin{lemma}
Let there exist $\delta>0$ and $z^*\in S_{X^*}$ such that $X$ denies the Daugavet property with $S(B_{X},z^*,\delta)\cap S_X$. Then $X$ is not a Daugavet domain.
 \label{NoDCfromDeny}
\end{lemma}

\begin{proof}
According to Definition~\ref{DD} we must prove that any $G \dopu X \to Y$ is not a Daugavet center for any $Y$. It is easy to see that if $G$ is a Daugavet center then $G/\|G\|$ is as well, so we consider only the case $\|G\|=1$.

Take the $\eps$ from item (ii) of Lemma~\ref{CharDeny}. At first we show that if every $z\in S(B_{X},z^*,\delta)$ satisfies $\|Gz\|\leq 1-\eps/2$ then $G$ is not a Daugavet center.
Put $\eps_0:=\min \{\eps/2, \delta\}$, then for every $y \in S_Y$ and every $z\in S(B_{X},z^*,\eps_0)$ we have 
$$\|y+Gz\| \leq 1 + \|Gz\| \leq 2-\frac{\eps}{2}\leq 2-\eps_0.$$ 
Theorem~\ref{charDC}, item (ii) implies that $G$ is not a Daugavet center.
 
So, we suppose that there is a $z_0\in S(B_{X},z^*,\delta)$ with 
\begin{equation}
\|Gz_0\|>1-\frac{\eps}{2}.
 \label{eq:1}
\end{equation}
We can assume $\|z_0\|=1$, because if $z_0 \in B_X$ fulfills $z^*(z_0)>1-\delta$ and~(\ref{eq:1}) then $z_0/\|z_0\|$ does as well. In addition,~(\ref{eq:1}) implies that there is a $y_0\in S_Y$ with $\|y_0-Gz_0\|<\eps/2$. Since $X$ denies the Daugavet property with $S(B_{X},z^*,\delta)\cap S_X$, there is an $x^*\in S_{X^*}$ such that every $x\in S(B_{X},x^*,\eps)$ satisfies $\|x+z_0\|<2-\eps$. Hence for every $x\in S(B_{X},x^*,\eps)$ we have
$$\|y_0+Gx\| \leq \|y_0-Gz_0\|+\|Gx+Gz_0\|<\frac{\eps}{2}+\|x+z_0\|<2-\frac{\eps}{2}.$$
By Theorem~\ref{charDC}, item (ii) $G$ is not a Daugavet center.
\end{proof}

\begin{lemma} 
Let there exist $\delta>0$ and $z\in S_{X}$ such that $X$ star-denies the Daugavet property with $S(B_{X^*},z,\delta)\cap S_{X^*}$. Then $X$ is not a Daugavet range.
\label{NoDCtoDenyAst}
\end{lemma}

Using item (iii) of Lemma~\ref{CharDenyAst} and item (iii) of Theorem~\ref{charDC} one can prove Lemma~\ref{NoDCtoDenyAst} in a very similar fashion to Lemma~\ref{NoDCfromDeny}.

\section{Two-dimensional lattices denying the positive Daugavet property}

\begin{definition}
We say that $F$ \textit{denies the positive Daugavet property with} $A\subset S_{F}^+$ if there is  an $\eps>0$ such that for every $a \in A$ there exists an $f^* \in S_{F^*}^+$ satisfying  \begin{equation}
\|\mathrm{Id}+f^*\otimes a\|<2-\eps.
\label{DaugEqFailF}
\end{equation}
 \label{LatticeDeny}
\end{definition}

\begin{definition} 
We say that $F$ \textit{star-denies the positive Daugavet property with} $A\subset S_{F^*}^+$ if there is an $\eps>0$ such that for every $f^* \in A$ there exists an $a \in S_{F}^+$ satisfying~(\ref{DaugEqFailF}).
\label{LatticeDenyAst} 
\end{definition}

The following two lemmas are complete analogs of Lemmas~\ref{CharDeny} and~\ref{CharDenyAst}, so we skip their proofs.

\begin{lemma} 
For $A\subset S_F^+$ the following assertions are equivalent:
\begin{aequivalenz}
\item $F$ denies the positive Daugavet property with $A$.
\item There is an $\eps>0$ such that for every $a\in A$ a functional $f^*\in S_{F^*}^+$ may be chosen so that every $b\in S(B_{F},f^*,\eps)\cap B_F^+$ fulfills $\|a+b\|<2-\eps$.  
\item There is an $\eps>0$ such that for every $a\in A$ a functional $f^*\in S_{F^*}^+$ may be chosen so that every $g^*\in S(B_{F^*},a,\eps)\cap B_{F^*}^+$ fulfills $\|f^*+g^*\|<2-\eps$. 
\end{aequivalenz}
\label{CharLatticeDeny}
\end{lemma}

\begin{lemma} 
For $A\subset S_{F^*}^+$ the following assertions are equivalent:
\begin{aequivalenz}
\item $F$ star-denies the positive Daugavet property with $A$.
\item There is an $\eps>0$ such that for every $f^*\in A$ a vector $a\in S_{F}^+$ may be chosen so that every $b\in S(B_{F},f^*,\eps)\cap B_F^+$ fulfills $\|a+b\|<2-\eps$. 
\item There is an $\eps>0$ such that for every $f^*\in A$ a vector $a\in S_{F}^+$ may be chosen so that every $g^*\in S(B_{F^*},a,\eps)\cap B_{F^*}^+$ fulfills $\|f^*+g^*\|<2-\eps$. 
\end{aequivalenz}
\label{CharLatticeDenyAst}
\end{lemma}

Recall that $F^*=\R^2$ with the norm $$\|(f_1,f_2)\|_{F^*}:=\max_{(a_1,a_2)\in B_F}  |f_1a_1+f_2a_2|$$ and $F^{**}=F$. We introduce an order on $F^*$ the same way as on $F$. It is easy to see that $\|(1,0)\|_{F^*}=\|(0,1)\|_{F^*}=1$ and $\|(f_1,f_2)\|_{F^*}=\|(|f_1|,|f_2|)\|_{F^*}$ for every $(f_1,f_2)\in F^*$. Hence $F^*$ is a two-dimensional lattice as well.
Lemmas~\ref{CharLatticeDeny} and~\ref{CharLatticeDenyAst} evidently imply the following fact (which one can easily deduce from Definitions~\ref{LatticeDeny} and~\ref{LatticeDenyAst} as well).

\begin{lemma} 
Let $A \subset S_{F}^+$ and $\tilde{A} \subset S_{F^*}^+$. 
\begin{statements}
\item If $F$ denies the positive Daugavet property with $A$ then $F^*$ star-denies the positive Daugavet property with~$A$.
\item If $F$ star-denies the positive Daugavet property with $\tilde{A}$ then $F^*$ denies the positive Daugavet property with~$\tilde{A}$.
\end{statements}
\label{Switch}
\end{lemma}

Here is the key lemma of this section. In its proof we use the idea from Theorem 5.1 of~\cite{BKSW}.
\begin{lemma} 
Let there exist $w^*\in S_{F^*}^+$ and $\delta>0$ such that $F$ denies the positive Daugavet property with $S(B_{F},w^*,\delta)\cap S_F^+$. Then $X_1\oplus_F X_2$ is not a Daugavet domain for any $X_1$ and $X_2$.
\label{CorFDeny}
\end{lemma}

\begin{proof}
It is easy to see that $(X_1\oplus_F X_2)^*=X_1^*\oplus_{F^*} X_2^*$ for every $X_1$ and $X_2$. Pick a $z^*=(z_1^*,z_2^*)\in S_{(X_1\oplus_F X_2)^*}$ with $(\|z_1^*\|,\|z_2^*\|)=w^*$. Then for a $y=(y_1,y_2)\in S(B_{X_1\oplus_F X_2},z^*,\delta)\cap S_{X_1\oplus_F X_2}$ we have 
$$\|z_1^*\|\|y_1\|+\|z_2^*\|\|y_2\|\geq z_1^*(y_1)+z_2^*(y_2)=z^*(y)>1-\delta.$$
Hence $a:=(\|y_1\|,\|y_2\|)\in S(B_{F},w^*,\delta)\cap S_F^+$. By item (ii) of Lemma~\ref{CharLatticeDeny} there exist $\eps>0$ and $f^* \in S_{F^*}^+$ such that every $b \in S(B_{F},f^*,\eps)\cap B_F^+$ satisfies $\|a+b\|<2-\eps$.

Pick an $x^*=(x_1^*,x_2^*)\in S_{(X_1\oplus_F X_2)^*}$  with $(\|x_1^*\|,\|x_2^*\|)=f^*$. Then for every $x=(x_1,x_2)\in S(B_{X_1\oplus_F X_2},x^*,\eps)$ we have $b_x:=(\|x_1\|,\|x_2\|)\in S(B_{F},f^*,\eps)\cap B_F^+$ and therefore
\begin{eqnarray*}
\|x+y\|&=&\|(\|x_1+y_1\|,\|x_2+y_2\|)\| \\
&\leq&\|(\|x_1\|+\|y_1\|,\|x_2\|+\|y_2\|)\|=\|a+b_x\|<2-\eps.
\end{eqnarray*}
By item (ii) of Lemma~\ref{CharDeny} $X_1\oplus_F X_2$ denies the Daugavet property for $S(B_{X_1\oplus_F X_2},z^*,\delta)\cap S_{X_1\oplus_F X_2}$. So, Lemma~\ref{NoDCfromDeny} implies that $X_1\oplus_F X_2$ is not a Daugavet domain. 
\end{proof}

The same conclusions based on item (iii) of Lemma~\ref{CharLatticeDenyAst}, item (iii) of Lemma~\ref{CharDenyAst}, and Lemma~\ref{NoDCtoDenyAst} prove the following fact:

\begin{lemma} 
Let there exist $w\in S_{F}^+$ and $\delta>0$ such that $F$ star-denies the positive Daugavet property with $S(B_{F^*},w,\delta)\cap S_{F^*}^+$. Then $X_1\oplus_F X_2$ is not a Daugavet range for any $X_1$ and $X_2$. 
\label{CorFDenyAst}
\end{lemma}

\section{Sums of spaces which are not Daugavet ranges}

In this section we find a large class of those $F$ which star-deny the positive Daugavet property with some $S(B_{F^*},w,\delta)\cap S_{F^*}^+$.
Throughout this and the following sections $e_1:=(1, 0)\in S_F^+$, $e_2:=(0, 1)\in S_F^+$, and the symbol $[a,b]$ is reserved for the line segment with the end points in $a, b \in F$.

\begin{lemma} 
Let $D$ be a closed subset of $S_{F^*}^+$. Suppose for every $f^* \in D$ there exists an $\eps > 0$ such that the property $P(f^*,\eps):=\{$there is an $a \in S_{F}^+$ such that every $b \in S(B_{F},f^*,\eps)\cap B_F^+$ satisfies $\|a+b\|<2-\eps \}$ holds true. Then $F$ star-denies the positive Daugavet property with $D$.
\label{QuantPermut2}
\end{lemma}

\begin{proof}
Note that if $P(f^*,\eps)$ holds true then $P(f^*,\eps_1)$ holds for every $\eps_1\dopu 0<\eps_1<\eps$. Our goal is to show that there exists a common $\eps_{min}>0$ such that $P(f^*,\eps_{min})$ holds true for every $f^*\in D$.

Consider the function $u(f^*)\dopu D\to(0,1)$, $u(f^*)=\sup\{\eps>0\,\dopu\;P(f^*,\eps)\;holds\;true\}$. Let us prove that $u(f^*)$ reaches its minimum value on $D$. Since $D$ is compact, it is sufficient to show that $u(f^*)$ is lower semicontinuous, i.e. that the set $u^{-1}\bigl((x,1)\bigr)$ is open for every $x \in [0,1)$. 

Let $f^*\in u^{-1}\bigl((x,1)\bigr)$. This means that $u(f^*)= \sup\{\eps>0\,\dopu\;P(f^*,\eps)\;holds\;true\}>x$. Hence there exist $\eps_0>x$ and $a\in S_{F}^+$ such that every $b\in S(B_{F},f^*,\eps_0)\cap B_F^+$ fulfills $\|a+b\|<2-\eps_0$.

Take an $\eps_1\dopu x<\eps_1<\eps_0$ and put $\delta:=\eps_0-\eps_1$. The set $D\cap B_{F^*}(f^*,\delta)$ is a relative neighborhood of $f^*$ in $D$. Let us show that $D\cap B_{F^*}(f^*,\delta)\subset u^{-1}\bigl((x,1)\bigr)$.

Let $f^*_1\in D\cap B_{F^*}(f^*,\delta)$. Then every $b\in S(B_{F},f_1^*,\eps_1)\cap B_F^+$ fulfills
$$f^*(b)\geq f_1^*(b)-\delta>1-\eps_1-\delta=1-\eps_0.$$ 
Thus $b\in S(B_{F},f^*,\eps_0)\cap B_F^+$, so we have
$$\|a+b\|<2-\eps_0<2-\eps_1.$$ 
This means that $u(f_1^*)\geq \eps_1>x$ and $f_1^*\in u^{-1}\bigl((x,1)\bigr)$. Consequently, $u^{-1}\bigl((x,1)\bigr)$ is open and $u(f^*)$ is lower semicontinuous. Then there exists an $f^*_0\in D$ such that $$u(f^*_0)=\min_{f^* \in D}u(f^*).$$
 
Take an $\eps_{min}\dopu 0<\eps_{min}<u(f^*_0)$ then $P(f^*,\eps_{min})$ holds true for every $f^*\in D$.
\end{proof} 

\begin{lemma} 
Let $a\in S_F^+$ and $f^*\in S_{F^*}^+$. Suppose for every $\eps>0$ there is a $b\in S(B_{F},f^*,\eps)\cap B_F^+$ with $\|a+b\| \geq 2-\eps$. Then there exists a $b_0\in S_F^+$ such that $f^*(b_0)=1$ and $[a,b_0]\subset S_F^+$. 
\label{Segment}
\end{lemma}

\begin{proof}
Consider a vanishing sequence $\{\eps_n\}_{n=1}^{\infty}$, $\eps_n>0$. For every $n\in\N$ there exists a $b_n\in B_F^+$ with $f^*(b_n)>1-\eps_n$ and $\|a+b_n\|\geq 2-\eps_n$.

Since $B_F^+$ is a compact set, there exists a subsequence $\{b_{n_i}\}_{i=1}^{\infty}$ of $\{b_n\}_{n=1}^{\infty}$ that converges to some $b_0\in B_F^+$. Then $f^*(b_0)=1$ and $\|a+b_0\|=2$ which implies $[a,b_0]\subset S_F^+$. 
\end{proof}

Denote $\mathfrak{N}_3$ the class of those $F$ whose $S_F^+$ is a polygon which consists of at most three edges.

\begin{lemma} 
Let $F \notin \mathfrak{N}_3$. Then $F$ star-denies the positive Daugavet property with $S_{F^*}^+$.
\label{lemPropNotInN}
\end{lemma}

\begin{proof}
Assume to the contrary that there exists an $f^*\in S_{F^*}^+$ such that for every $\eps>0$ and $a_0\in S_{F}^+$ there is a $b\in S(B_{F},f^*,\eps)\cap B_F^+$ with $\|a_0+b\|\geq 2-\eps$.

Consider the set $\Delta:=\{a\in S_{F}^+: f^*(a)=1\}$. It is easy to see that $\Delta$ is a segment or a point. Put $a_0:=e_1$. By Lemma~\ref{Segment} there exists a $b_0 \in \Delta$ such that $[b_0,e_1]\subset S_{F}^+$. If we put $a_0:=e_2$ we obtain a $b_1 \in \Delta$ with $[b_1,e_2]\subset S_{F}^+$. Then $F \in \mathfrak{N}_3$, because $S_{F}^+$ consists of at most three segments: $[b_0,e_1]$, $\Delta$ and $[b_1,e_2]$. This contradiction completes the proof.
\end{proof}

\begin{lemma} 
Let $S_F^+$ be a polygon which consists of exactly three edges. Then there exists a $w^*=(w_1,w_2)\in S_{F^*}^+$ with $w_1<1$ and $w_2<1$ such that $F$ star-denies the positive Daugavet property with $S_{F^*}^+\setminus \stackrel{_\circ}{B}_{F^*}(w^*, \delta_0)$ for every $\delta_0>0$.
\label{noPropSomeF2}
\end{lemma}

\begin{proof}
Since $S_F^+$ consists of three edges, it has four vertexes. The points $e_1$ and $e_2$ are two of them, denote $h_1$ and $h_2$ the remaining ones in such a way that $[e_1,h_1]\cup[e_2,h_2]\subset S_{F}^+$. There is the unique $w^*=(w_1,w_2)\in S_{F^*}^+$ such that $\{a \in S_F^+: w^*(a)=1\}=[h_1,h_2]$. It is obvious that $w_1<1$ and $w_2<1$.

Consider a $\delta_0>0$ and an $f^* \in S_{F^*}^+\setminus \stackrel{_\circ}{B}_{F^*}(w^*, \delta_0)$. Denote $\Delta:=\{a \in S_F^+: f^*(a)=1\}$, it is a segment or a point. Assume that for every $\eps >0$ and $a\in S_F^+$ there exists a $b\in S(B_{F},f^*,\eps)\cap B_F^+$ with $\|a+b\|\geq 2-\eps$. By Lemma~\ref{Segment} there are $b_1, b_2 \in \Delta$ such that $[b_1,e_1]\subset S_{F}^+$ and $[b_2,e_2]\subset S_{F}^+$. Hence $b_1\in[e_1,h_1]$ and $b_2\in[e_2,h_2]$. Since $[e_1,h_1]\cap[e_2,h_2]=\emptyset$ then $\Delta\nsubseteq[e_1,h_1]$, $\Delta\nsubseteq[e_2,h_2]$, and $\Delta$ is a segment. Consequently, $\Delta\subset[h_1,h_2]$. But then $w^*=f^*$, so we come to contradiction.

Thus for every $f^*\in S_{F^*}^+\setminus \stackrel{_\circ}{B}_{F^*}(w^*, \delta_0)$ there are $\eps>0$ and $a\in S_{F}^+$ such that every $b\in S(B_{F},f^*,\eps)\cap B_F^+$ satisfies $\|a+b\|<2-\eps$. Since $S_{F^*}^+\setminus \stackrel{_\circ}{B}_{F^*}(w^*, \delta_0)$ is closed, Lemma~\ref{QuantPermut2} implies the needed result.
\end{proof} 

Denote $\mathfrak{N}_2$ the class of those $F$ whose $S_F^+$ is a polygon which consists of at most two edges.

\begin{corollary} 
Let $F\notin\mathfrak{N}_2$. Then there is a $\delta>0$ such that $F$ star-denies the positive Daugavet property with $S(B_{F^*},e_1,\delta)\cap S_{F^*}^+$.
\label{notN2FDeny}
\end{corollary}

\begin{proof}
If $F\notin \mathfrak{N}_3$ then by Lemma~\ref{lemPropNotInN} the statement is proved.

If $S_F^+$ is a polygon which consists of exactly three edges then by Lemma~\ref{noPropSomeF2} there exists a $w^*=(w_1,w_2)\in S_{F^*}^+$ with $w_1<1$ such that $F$ star-denies the positive Daugavet property with $S_{F^*}^+\setminus \stackrel{_\circ}{B}_{F^*}(w^*, \delta_0)$ for every $\delta_0>0$. Pick a $\delta_0>0$ with $\delta_0+w_1<1$ and a $\delta$ such that $0<\delta<1-w_1-\delta_0$. Then every $f^*\in \stackrel{_\circ}{B}_{F^*}(w^*,\delta_0)$ satisfies
$$f^*(e_1)<w^*(e_1)+\delta_0=w_1+\delta_0<1-\delta.$$
Hence $\stackrel{_\circ}{B}_{F^*}(w^*,\delta_0)\cap S(B_{F^*},e_1,\delta)=\emptyset$. Thus $F$ star-denies the positive Daugavet property with $S(B_{F^*},e_1,\delta)\cap S_{F^*}^+$.
\end{proof}

We obtain the following fact by the successive application of Corollary~\ref{notN2FDeny} and Lemma~\ref{CorFDenyAst}.

\begin{corollary} 
Let $F\notin \mathfrak{N}_2$. Then $X_1\oplus_F X_2$ is not a Daugavet range for any $X_1$ and $X_2$.
\label{noDCtoSomeF}
\end{corollary}

\section{Sums of spaces which are not Daugavet domains}

\begin{lemma} 
Let $F^*\notin \mathfrak{N}_2$. Then $X_1 \oplus_F X_2$ is not a Daugavet domain for any $X_1$ and $X_2$.
\label{noDCfromSomeF}
\end{lemma}

\begin{proof}
By Corollary~\ref{notN2FDeny} there is a $\delta>0$ such that $F^*$ star-denies the positive Daugavet property with $S(B_{F^{**}},e_1,\delta)\cap S_{F^{**}}^+$. Recall that $F^{**}=F$. Therefore it follows from Lemma~\ref{Switch} that $F$ denies the positive Daugavet property with $S(B_{F},e_1,\delta)\cap S_{F}^+$. Then Lemma~\ref{CorFDeny} gives the needed result. 
\end{proof}

We characterize the class of those $F$ such that $S_{F^*}^+$ is a polygon with at most two edges, with the help of the following notation.
Consider an $F$ whose $S_F^+$ is a polygon with $n$ edges. Denote $\hat{x}_1:=\max_{(1, y) \in S_F^+} y$ and $\hat{x}_2:=\max _{(x, 1) \in S_F^+} x$. We say that $F$ belongs to $\mathcal{F}_{n-1,n}$ if $\hat{x}_1 > 0$ and $\hat{x}_2 > 0$. If only one of $\hat{x}_j$ equals zero, we say that $F\in\mathcal{F}_{n,n}$. And if both $\hat{x}_1=\hat{x}_2=0$ then $F\in\mathcal{F}_{n+1,n}$ (see Figure~\ref{Fig1}).
\begin{figure}[h]
\includegraphics[width=0.22\textwidth]{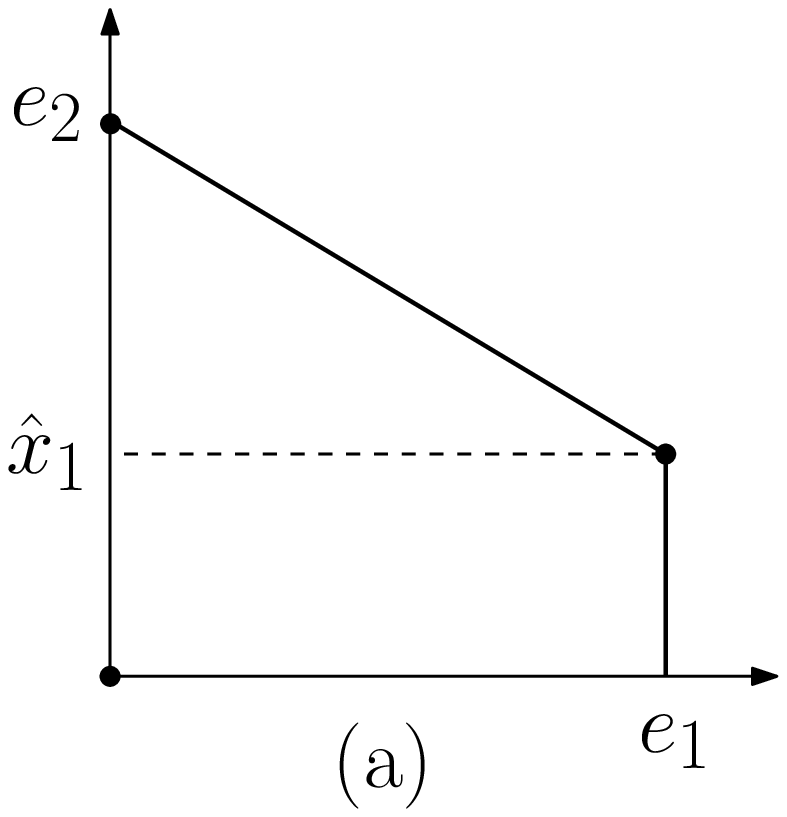}
\includegraphics[width=0.22\textwidth]{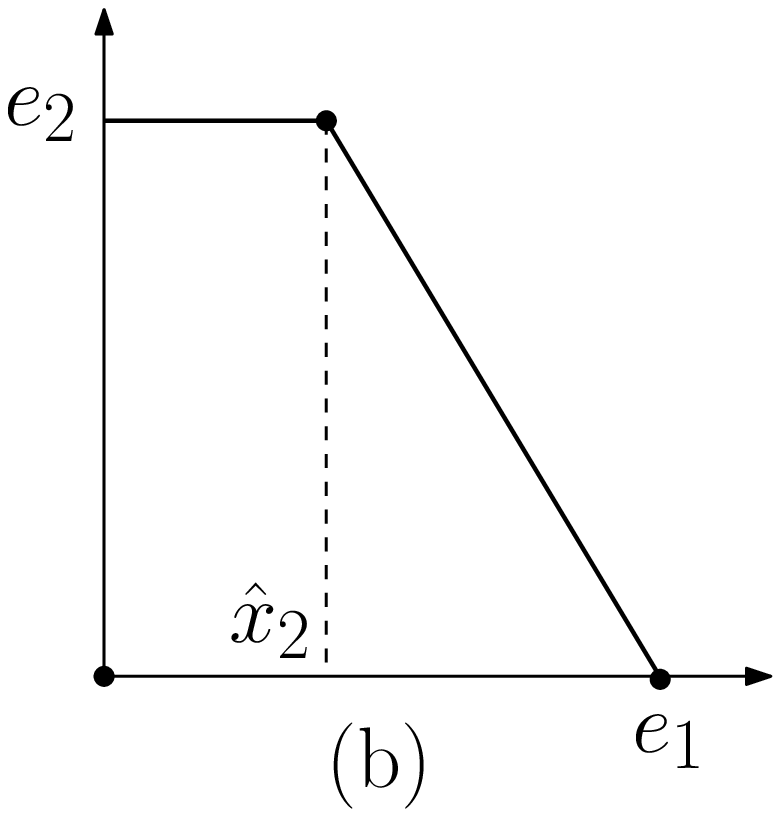}
\includegraphics[width=0.22\textwidth]{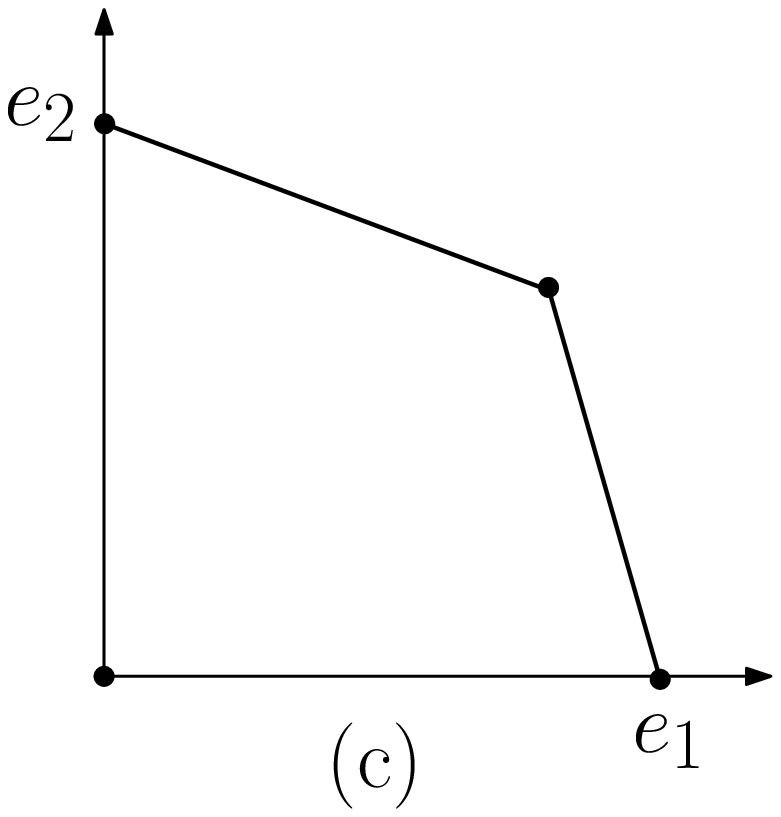}
\includegraphics[width=0.22\textwidth]{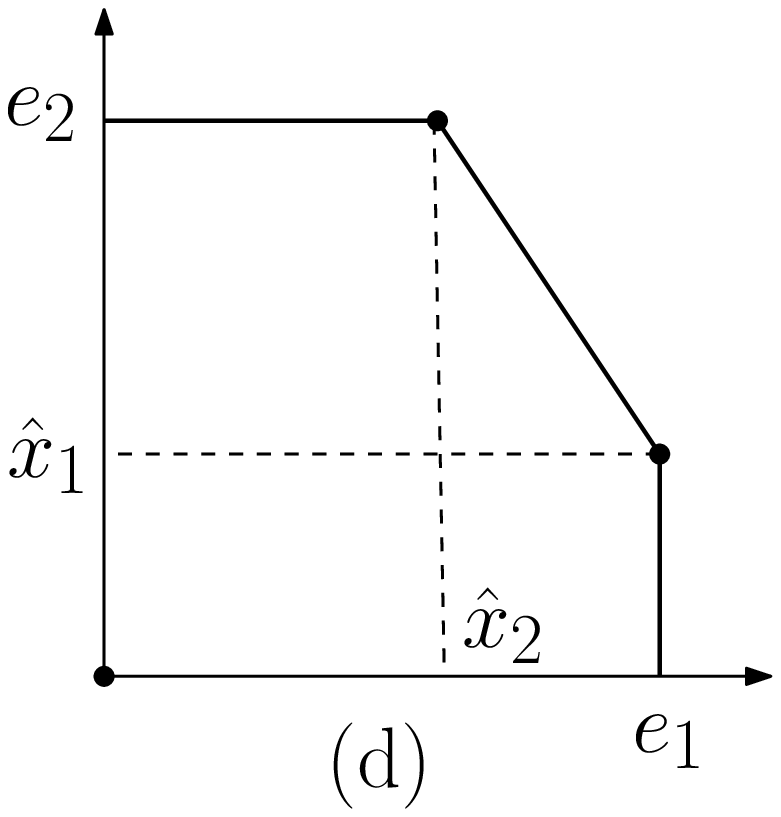}
\caption{Those $F$ whose $S_F^+$ are presented on pictures (a) and (b), belong to $\mathcal{F}_{2,2}$. Picture (c) shows $S_F^+$ of $F\in\mathcal{F}_{3,2}$ and (d) shows $S_F^+$ of $F\in\mathcal{F}_{2,3}$.}
\label{Fig1}
\end{figure}

Thus, $\mathfrak{N}_2=\{\ell_{1}^{\scriptscriptstyle (2)}\}\cup\{\ell_{\infty}^{\scriptscriptstyle (2)}\}\cup\mathcal{F}_{2,2}\cup\mathcal{F}_{3,2}$. Let $n\in\N$ and $m\in\{n-1,n,n+1\}$. It is easy to see that $F^*\in \mathcal{F}_{n,m}$ if and only if $F\in \mathcal{F}_{m,n}$. Therefore, if $F^*\in \mathcal{F}_{2,2}\cup\mathcal{F}_{3,2}$ then $F\in \mathcal{F}_{2,2} \cup \mathcal{F}_{2,3}$. So, we obtain the following fact:

\begin{corollary} 
Let $F\notin \{\ell_1^{\scriptscriptstyle (2)}\}\cup\{\ell_{\infty}^{\scriptscriptstyle (2)}\}\cup\mathcal{F}_{2,2} \cup \mathcal{F}_{2,3}=:\mathfrak{M}_2$. Then $X_1 \oplus_F X_2$ is not a Daugavet domain for any $X_1$ and $X_2$.
\label{CorNoDCfromF}
\end{corollary}

\section{Examples of Daugavet centers acting from and into a sum of two Banach spaces}

In this section we show that for every $F\in\mathfrak{M}_2$ there exists a Daugavet domain $X_1\oplus_FX_2$, and for every $F\in\mathfrak{N}_2$ there is a Daugavet range $X_1\oplus_FX_2$.

For $F=\ell_{1}^{\scriptscriptstyle (2)}$ and $F=\ell_{\infty}^{\scriptscriptstyle (2)}$ several examples of $X_1\oplus_FX_2$ which are Daugavet domains and Daugavet ranges, are known. For instance, if $X$ is a Daugavet domain then  for every $E$ the sum $X\oplus_{\infty}E$ is as well; and if $X$ is a Daugavet range then $X\oplus_{1}E$ is. If $G_1 \dopu X_1 \to Y_1$ and $G_2 \dopu X_2 \to Y_2$ are Daugavet centers then $G\dopu X_1 \oplus_1 X_2 \to Y_1 \oplus_1 Y_2$ and $\tilde{G}\dopu X_1 \oplus_{\infty} X_2 \to Y_1 \oplus_{\infty} Y_2$ which map every $(x_1, x_2)$ into $(G_1x_1, G_2x_2)$, are Daugavet centers as well~\cite{BosKad}.

For future reference we mention the following fact:

\begin{lemma}[\cite{KadSSW}, Lemma 2.8] 
If $X$ has the Daugavet property then for every finite-dimensional subspace
$Y_{0}$ of $X$, every $\eps>0$, and every slice $S(B_{X},x^*,\eps)$ there is an $x\in S(B_{X},x^*,\eps)$ such that every $y\in Y_{0}$ and $t\in\R$ fulfill
$$\|y+tx\|\ge(1-\eps)(\|y\|+|t|).$$
\label{DPlin}
\end{lemma}

Consider an $F \in \mathcal{F}_{2,2} \cup \mathcal{F}_{2,3}$. Denote $c_1:=(1, \hat{x}_1) \in S_F^+$ and $c_2:=(\hat{x}_2, 1) \in S_F^+$. Then $[c_1,c_2]\subset S_F^+$ (see Figure~\ref{Fig2}). 
\begin{figure}[h]
\includegraphics[width=0.22\textwidth]{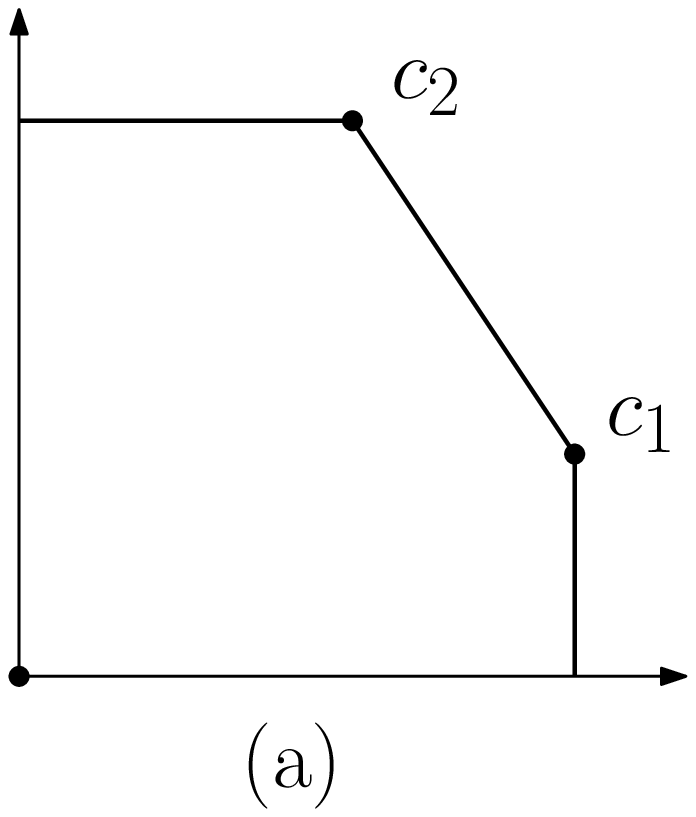}
\includegraphics[width=0.22\textwidth]{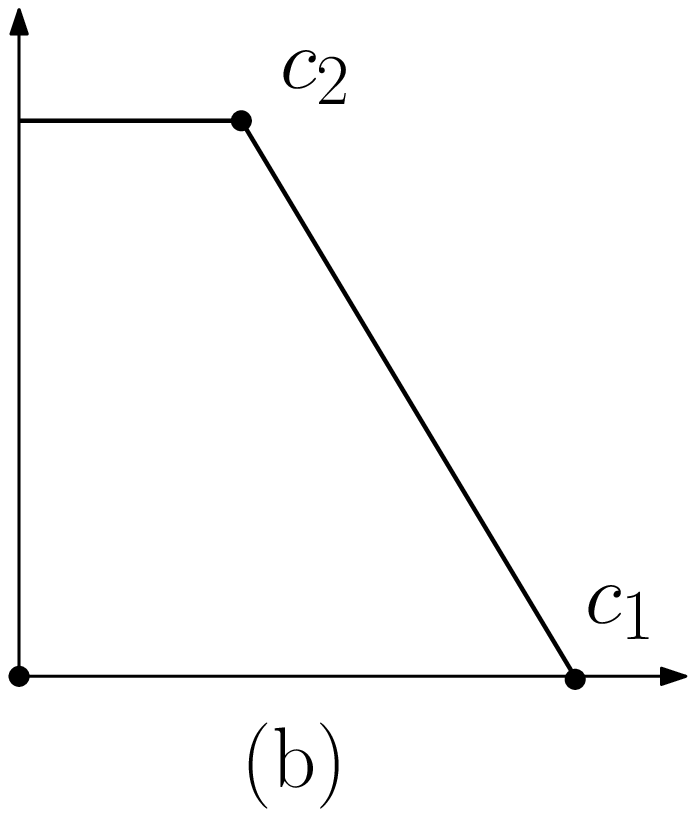}
\caption{Picture (a) shows $S_F^+$ of $F\in\mathcal{F}_{2,3}$ and (b) presents $S_F^+$ of $F\in\mathcal{F}_{2,2}$.}
\label{Fig2}
\end{figure}

Let the line containing $[c_1,c_2]$ have the equation $f_1a_1+f_2a_2=1$ with $(f_1, f_2)\in S_{F^*}^+$. Remark that for every $w^*\in S_{F^*}^+$ and $\eps>0$ we have $S(B_{F},w^*,\eps)\cap[c_1,c_2]\neq\emptyset$. In other words, there exists an $(a_1, a_2)\in S(B_{F},w^*,\eps)$ such that $f_1a_1+f_2a_2=1$.

\begin{proposition} 
Let $X$ have the Daugavet property, $F\in\mathcal{F}_{2,2}\cup\mathcal{F}_{2,3}$, and let $(f_1, f_2) \in S_{F^*}^+$ be the functional described above. Then $G \dopu X \oplus_F X \to X$, $G(x_1, x_2)=f_1x_1+f_2x_2$ is a Daugavet center.
\label{DCfromExmp}
\end{proposition}

\begin{proof}
At first, calculate $\|G\|$:
$$\|G\|=\sup_{(x_1, x_2)\in S_{X\oplus_F X}}\|f_1x_1+f_2x_2\|=\sup_{(x_1,x_2)\in S_{X\oplus_F X}}(f_1\|x_1\|+f_2\|x_2\|)=1.$$
Let $\eps>0$, $y_0 \in S_X$ and $x^*=(x_1^*, x_2^*)\in S_{(X \oplus_F X)^*}$.

By Lemma~\ref{DPlin} there exists an $\tilde{x}_1\in B_X$ with $x_1^*(\tilde{x}_1) \geq \|x_1^*\|(1-\eps/4)$ and 
\begin{equation} 
\|y_0+t\tilde{x}_1\|\ge\left(1-\frac{\eps}{4}\right)(1+|t|)
\label{DP1}
\end{equation} 
for every $t\in\R$.
Using again Lemma~\ref{DPlin} we have an $\tilde{x}_2\in B_X$ with $x_2^*(\tilde{x}_2) \geq \|x_2^*\|(1-\eps/4)$ and 
\begin{equation} 
\|y+t\tilde{x}_2\|\ge\left(1-\frac{\eps}{4}\right)(\|y\|+|t|)
\label{DP2}
\end{equation}
for every $y\in\mathop{\rm lin}\nolimits\{y_0,\tilde{x}_1\}$ and every $t\in\R$. 

Denote $w^*:=(\|x_1^*\|,\|x_2^*\|)\in S_{F^*}^+$. Let $(a_1, a_2)\in S(B_{F},w^*,3\eps/4)$ such that $f_1a_1+f_2a_2=1$. Then for $x:=(a_1\tilde{x}_1,a_2\tilde{x}_2)\in B_{X\oplus_F X}$ we have
\begin{eqnarray*}
x^*(x)~=~a_1x_1^*(\tilde{x}_1)+a_2x_2^*(\tilde{x}_2) &\geq&\left(1-\frac{\eps}{4}\right)(a_1\|x_1^*\|+a_2\|x_2^*\|)\\ &\geq&  \left(1-\frac{\eps}{4}\right)\left(1-\frac{3\eps}{4}\right)~>~1-\eps.
\end{eqnarray*}
Hence $x\in S(B_{X \oplus_F X},x^*,\eps)$ and 
$$\|y_0+Gx\|=\|y_0+f_1a_1\tilde{x}_1+f_2a_2\tilde{x}_2\|$$
by~(\ref{DP2}) $$>\left(1-\frac{\eps}{4}\right)(\|y_0+f_1a_1\tilde{x}_1\|+f_2a_2)$$
by~(\ref{DP1}) $$>\left(1-\frac{\eps}{4}\right)^2(1+f_1a_1+f_2a_2)=2\left(1-\frac{\eps}{4}\right)^2>2-\eps.$$
Theorem~\ref{charDC}, item (ii) implies that $G$ is a Daugavet center. 
\end{proof}

\begin{corollary}
For an $F$ there exists a Daugavet domain $X_1\oplus_F X_2$ if and only if $F\in \mathfrak{M}_2$.
\label{corDD}
\end{corollary}

\begin{remark}
Note that $\mathfrak{M}_2\nsubseteq \mathfrak{N}_2$. Then Corollary~\ref{corDD} and Corollary~\ref{noDCtoSomeF} imply that there exist Daugavet domains which are not Daugavet ranges.
\end{remark}

Now we present more examples of Daugavet centers acting from $X_1\oplus_{F} X_2$ where  $F=\ell_{1}^{\scriptscriptstyle (2)}$ or $F=\ell_{\infty}^{\scriptscriptstyle (2)}$. 

\begin{proposition} 
Let $X$ have the Daugavet property. Then 
\begin{statements} 
\item The operator $G \dopu X \oplus_1 X \to X$, $G(x_1, x_2)=x_1+x_2$ is a Daugavet center.
\item For every $f_1, f_2>0$ the operator $G \dopu X \oplus_{\infty} X \to X$, $G(x_1, x_2)=f_1x_1+f_2x_2$ is a Daugavet center.
\end{statements}
\label{ells}
\end{proposition}

Proposition~\ref{ells} can be proved the same way as Proposition~\ref{DCfromExmp}.

\begin{proposition} 
Let $X$ have the Daugavet property, $F \in \mathcal{F}_{2,2} \cup \mathcal{F}_{3,2}$, and let $(f_1, f_2) \in S_{F}^+$ be the vector described above. Then $G \dopu X \to X \oplus_F X$, $Gx=(f_1x, f_2x)$ is a Daugavet center.
\label{DCtoExmp}
\end{proposition}

\begin{proof}
Consider the adjoint operator $G^* \dopu X^* \oplus_{F^*} X^* \to X^*$. For every $(x^*_1, x^*_2) \in X^* \oplus_{F^*} X^*$ and every $x \in X$ we have $$G^*(x^*_1, x^*_2)(x)=\left\langle(f_1x, f_2x), (x^*_1, x^*_2) \right\rangle=f_1x^*_1(x)+f_2x^*_2(x).$$
Consequently, $G^*(x^*_1, x^*_2)=f_1x^*_1+f_2x^*_2$ for every $(x^*_1, x^*_2) \in X^* \oplus_{F^*} X^*$. 
By Proposition~\ref{DCfromExmp} $G^*$ is a Daugavet center. The equation~(\ref{eqDC}) implies that if $G^*$ is a Daugavet center then $G$ is as well.
\end{proof}

\begin{corollary}
For an $F$ there exists a Daugavet range $X_1\oplus_F X_2$ if and only if $F\in \mathfrak{N}_2$.
\end{corollary}

\begin{remark}
Since $\mathfrak{N}_2\nsubseteq \mathfrak{M}_2$, we have the examples of Daugavet ranges which are not Daugavet domains.
\end{remark}

Proposition~\ref{ells2} which gives more examples of Daugavet centers acting into $X_1\oplus_{F} X_2$ for $F=\ell_{1}^{\scriptscriptstyle (2)}$ and $F=\ell_{\infty}^{\scriptscriptstyle (2)}$,  can be proved in a very similar way to Proposition~\ref{DCtoExmp}.

\begin{proposition} 
Let $X$ have the Daugavet property. Then 
\begin{statements} 
\item The operator $G \dopu X \to X \oplus_{\infty} X$, $Gx=(x, x)$ is a Daugavet center.
\item For every $f_1, f_2>0$ the operator $G \dopu X \to X \oplus_{1} X$, $Gx=(f_1x, f_2x)$ is a Daugavet center.
\end{statements}
\label{ells2}
\end{proposition}

\section*{Acknowledgements}

The author is grateful to her scientific adviser Vladimir M. Kadets for attention and numerous fruitful discussions, and to the referee for useful remarks.

\end{document}